\documentclass[10pt]{amsart}
\usepackage{amsmath, amsthm, amssymb,color,comment,csquotes,enumerate,fancyhdr,filecontents,graphicx,stmaryrd}
\usepackage[colorlinks=true,linkcolor=MidnightBlue,citecolor=ForestGreen,bookmarks=false]{hyperref}
\usepackage[usenames,dvipsnames]{xcolor}
\usepackage{tikz}
\usetikzlibrary{arrows,backgrounds,calc,fit,positioning,shapes} 

\newtheorem{thrm}{Theorem}[section]
\theoremstyle{definition}

\newtheorem{cor}[thrm]{Corollary}
\newtheorem{lmm}[thrm]{Lemma}
\newtheorem{propo}[thrm]{Proposition}

\newtheorem{defn}[thrm]{Definition}

\begin{document}

\title[A generalized characterization of algorithmic probability]{A generalized characterization \\ of algorithmic probability}
\author{Tom F. Sterkenburg}
\date{\today}
\address{Algorithms \& Complexity Group, Centrum Wiskunde \& Informatica, Amsterdam; Faculty of Philosophy, University of Groningen}
\thanks{This research was supported by \textsc{NWO} Vici project 639.073.904. I am grateful to Alexander Shen for valuable comments on an earlier version of this paper, to Peter Gr\"unwald, Jan Leike, and Dani\"el Noom for helpful discussions, and to Jeanne Peijnenburg for the question that initiated this work.}

\email{tom@cwi.nl}

\begin{abstract}
An a priori semimeasure (also known as ``algorithmic probability'' or ``the Solomonoff prior'' in the context of inductive inference) is defined as the transformation, by a given universal monotone Turing machine, of the uniform measure on the infinite strings. It is shown in this paper that the class of a priori semimeasures can equivalently be defined as the class of transformations, by all compatible universal monotone Turing machines, of any continuous computable measure in place of the uniform measure. Some consideration is given to possible implications for the prevalent association of algorithmic probability with certain foundational statistical principles. 
\end{abstract}

\maketitle

\section{Introduction}\label{sec:intro}
Levin \cite{ZvoLev70rms} first considered the transformation of the uniform measure $\lambda$ on the infinite bit strings by a universal monotone machine $U$. This transformation $\lambda_U$ is the function that for each finite bit string returns the probability that the string is generated by machine $U$, when $U$ is supplied a stream of uniformly random input (produced by tossing a fair coin, say). Levin attached to $\lambda_U$ the interpretation of an ``a priori probability'' distribution, because $\lambda_U$ dominates every other semicomputable semimeasure and so the initial assumption that a sequence is randomly generated from $\lambda_U$ is in an exact sense the weakest of randomness assumptions. 

Earlier on, Solomonoff \cite{Sol64ftii} described in a somewhat less precise way a very similar definition. His motivation was an ``a priori probability'' distribution to serve as an objective starting point in inductive inference. In this context the definition is known under various headers, including ``the Solomonoff prior'' and ``algorithmic probability;'' and it has been associated with certain foundational principles from statistics, to explain or support its merits as an idealized inductive method. %(Cf.\ \cite{Sol97jcss,LiVit08kca,Hut07tcs}.)

As commonly presented, however, the association with two main such principles (firstly, the principle of \emph{indifference}, and secondly, the principle of \emph{Occam's razor}) seems to essentially rest on the definition of $\lambda_U$ as a universal transformation of the \emph{uniform measure} $\lambda$. 

This raises the question whether the \emph{a priori semimeasures} (as we will call the functions $\lambda_U$ here) must be defined, as they always are, as the universal transformations of the uniform measure, or that the a priori semimeasures can equivalently be defined as universal transformations of other computable measures.

The main result of this paper is that any a priori semimeasure can indeed be obtained as a universal transformation of \emph{any} continuous computable measure. That is, for any continuous computable measure, an a priori semimeasure can equivalently be defined as giving the probabilities for finite strings being generated by a universal machine that is presented with a stream of bits sampled from \emph{this} measure. More precisely, for any continuous computable measure $\mu$, it is shown that the class of functions $\lambda_U$ for all universal monotone machines $U$ coincides with the class of functions $\mu_U$ (i.e., the transformation by $U$ of $\mu$) for all ($\mu$-compatible) universal machines $U$.

This work will be done in Section \ref{sec:contap}. First, in the current section, we cover basic notions and notation (Subsection \ref{subsec:nota}), discuss the characterization of the semicomputable semimeasures as the transformations via monotone machines of a continuous computable measure (Subsection \ref{subsec:contsemi}), and the analogous characterization for  semicomputable discrete semimeasures and prefix-free machines (Subsection \ref{subsec:discrsemi}).

\subsection{Basic notions and notation}\label{subsec:nota} 
%Let us first briefly cover some basic notions and notation.

\subsubsection*{Bit strings} Let $\mathbb{B}:=\{ 0,1\}$ denote the set of bits; $\mathbb{B}^*$ the set of all finite bit strings; $\mathbb{B}^n$ the set of bit strings $\sigma$ of length $|\sigma|=n$; $\mathbb{B}^{\leq n}$ the set of bit strings $\sigma$ of length $|\sigma|\leq n$; $\mathbb{B}^\omega$ the class of all infinite bit strings. The empty string is $\epsilon$. The concatenation of bit strings $\sigma$ and $\tau$ is written $\sigma\tau$; we write $\sigma \preccurlyeq \tau$ if $\sigma$ is an \emph{initial segment} of $\tau$ (so there is a $\rho$ such that $\sigma\rho = \tau$; we write $\sigma \prec \tau$ if $\rho \neq \epsilon$). The initial segment of $\sigma$ of length $n \leq |\sigma|$ is denoted $\sigma \upharpoonright_n$; the initial segment $\sigma \upharpoonright_{|\sigma|-1}$ is denoted $\sigma^-$. Strings $\sigma$ and $\tau$ are \emph{comparable}, $\sigma \sim \tau$, if $\sigma \preccurlyeq \tau$ or $\tau \prec \sigma$; if $\sigma$ and $\tau$ are not comparable we write $\sigma \mid \tau$. 

For given finite string $\sigma$, the class $\llbracket \sigma \rrbracket := \{ \sigma X : X \in \mathbb{B}^\omega\} \subseteq \mathbb{B}^\omega$ is the class of infinite extensions of $\sigma$. Likewise, for $A \subseteq \mathbb{B}^*$, let $\llbracket A \rrbracket := \{ \sigma X: \sigma \in A, X \in \mathbb{B}^\omega\}$.

\subsubsection*{Computable measures}
A probability measure over the infinite strings is generated by a \emph{premeasure}, a function $m: \mathbb{B}^* \rightarrow [0,1]$ that satisfies
\begin{enumerate}
\item $m(\epsilon) = 1$;
\item $m(\sigma0)+m(\sigma1 )= m(\sigma)$ for all $\sigma \in \mathbb{B}^*$.
\end{enumerate}
A premeasure $m$ gives rise to an \emph{outer measure} $\mu^*_m:\mathcal{P}(\mathbb{B}^\omega) \rightarrow [0,1]$ by 
$$\mu^*_m(\mathcal{A})= \inf \left\{ \sum_{\sigma \in A} m(\sigma): \mathcal{A} \subseteq \llbracket A \rrbracket \right\}.$$
By restricting $\mu^*_m$ to the $\mu$-measurable sets, i.e., the sets $\mathcal{A}\subseteq \mathbb{B}^\omega$ such that $\mu^*(\mathcal{B})=\mu^*(\mathcal{B} \cap \mathcal{A} )+\mu^*(\mathcal{B} \setminus \mathcal{A})$ for all $\mathcal{B} \subseteq \mathbb{B}^\omega$, we finally obtain the corresponding \emph{(probability) measure} $\mu_m$, that satisfies $\mu_m(\llbracket \sigma \rrbracket)=m(\sigma)$ for all $\sigma \in \mathbb{B}^*$. 

%Then the corresponding \emph{(probability) measure} $\mu_m$ satisfies $\mu_m(\llbracket \sigma \rrbracket)=m(\sigma)$ for each $\sigma \in \mathbb{B}^*$. 
The \emph{uniform (Lebesgue) measure} $\lambda$ is given by the premeasure $m$ with $m(\sigma)=2^{-|\sigma|}$ for all $\sigma \in \mathbb{B}^*$. A measure $\mu$ is \emph{nonatomic} or \emph{continuous} if there is no $X \in \mathbb{B}^\omega$ with $\mu(\{X\})>0$.

We call a total real-valued function $f: \mathbb{B}^* \rightarrow \mathbb{R}$ \emph{computable} if its values are uniformly computable reals: there is a computable $g: \mathbb{B}^* \times \mathbb{N} \rightarrow \mathbb{Q}$ such that $|g(\sigma,k)-f(\sigma)|<2^{-k} $ for all $\sigma,k$. This allows us to talk about computable premeasures. A measure $\mu$ we then call computable if $\mu=\mu^*_m$ for a computable premeasure $m$.

\subsubsection*{Semicomputable semimeasures}
We call a total real-valued function $f: \mathbb{B}^* \rightarrow \mathbb{R}$ (\emph{lower}) \emph{semicomputable} if there are uniformly computable functions $f_t: \mathbb{B}^* \rightarrow \mathbb{Q}$ such that  for all $\sigma \in \mathbb{B}^*$, we have $f_{t+1}(\sigma)\geq f_t(\sigma)$ for all $t \in \mathbb{N}$ and $\lim_{t \rightarrow \infty} f_t(\sigma)=f(\sigma)$.

Levin \cite[Definition 3.6]{ZvoLev70rms} introduced the notion of a semicomputable measure over the collection $\mathbb{B}^* \cup \mathbb{B}^\omega$ of finite and infinite strings. This is equivalent to a \emph{semimeasure} over the infinite strings that is generated from a premeasure $m$ that only needs to satisfy
\begin{enumerate}
\item $m(\epsilon) \leq 1$;
\item $m(\sigma0)+m(\sigma1 ) \leq m(\sigma)$ for all $\sigma \in \mathbb{B}^*$.
\end{enumerate}

Following \cite{DowHir10}, we will simply treat a semimeasure as a function over the cones $\{ \llbracket \sigma \rrbracket: \sigma \in \mathbb{B}^* \}$:
%TOCS \begin{definition}
\begin{defn}
A \emph{semicomputable semimeasure} is a function $\nu: \{ \llbracket \sigma \rrbracket: \sigma \in \mathbb{B}^* \} \rightarrow [0,1]$ such that $\nu(\llbracket \cdot \rrbracket):\mathbb{B}^* \rightarrow [0,1]$ is semicomputable, and
\begin{enumerate}
\item $\nu(\llbracket \epsilon \rrbracket)\leq 1$;
\item $\nu(\llbracket \sigma0 \rrbracket)+\nu(\llbracket \sigma1 \rrbracket)\leq \nu(\llbracket \sigma \rrbracket)$ for all $\sigma \in \mathbb{B}^*$.
\end{enumerate}
%TOCS \end{definition}
\end{defn}

Moreover, we follow the custom of writing $\nu(\sigma)$ for $\nu(\llbracket \sigma \rrbracket)$. Let $\mathcal{M}$ denote the class of all semicomputable semimeasures.\footnote{Semimeasures as defined here are often referred to as \emph{continuous} semimeasures, in contradistinction to the \emph{discrete} semimeasures defined in Subsection \ref{subsec:discrsemi} below (cf.\ \cite{LiVit08,DowHir10}). Due to the possibility of confusion with the earlier meaning of ``continuous'' as synonymous to ``nonatomic,'' we will avoid this usage here.}

\subsection{Monotone machines and semicomputable semimeasures}\label{subsec:contsemi}

\subsubsection*{Machines}
The following definition is due to Levin \cite{Lev73smd}. (Similar machine models were already described in \cite{ZvoLev70rms}, and by Solomonoff \cite{Sol64ftii} and Schnorr \cite{Sch73jcss}; see \cite{Day09apal}.) %(See \cite[336]{LiVit08kca}; \cite[145-6]{DowHir10arc}; \cite{Day09cprs}.)

%TOCS \begin{definition}
\begin{defn}
A \emph{monotone machine} is a c.e.\ set $M \subseteq \mathbb{B}^*\times \mathbb{B}^*$ of pairs of strings such that if $(\rho_1,\sigma_1),(\rho_2,\sigma_2)\in M$ and $\rho_1 \preccurlyeq \rho_2$ then $\sigma_1 \sim \sigma_2$. 
%TOCS \end{definition}
\end{defn}

We will not go into the concrete machine model that corresponds to the above abstract definition (see, for instance, \cite[p.\ 145]{DowHir10}); we only note that a machine $M$ as defined above induces a function $N_M: \mathbb{B}^* \cup \mathbb{B}^\omega \rightarrow \mathbb{B}^* \cup \mathbb{B}^\omega$ by $N_M(X) = \sup_{\preccurlyeq}\{ \sigma \in \mathbb{B}^*: \exists \rho \preccurlyeq X \left( (\rho,\sigma) \in M \right) \}$ (cf.\ \cite{Gac08unpubl}).

\subsubsection*{Transformations}
Imagine that we feed a monotone machine $M$ a stream of input that is generated from a computable measure $\mu$. As a result, machine $M$ produces a (finite or infinite) stream of output. The probabilities for the possible initial segments of the output stream are themselves given by a semicomputable semimeasure (as can easily be verified). We will call this semimeasure the \emph{transformation} of $\mu$ by $M$.

%TOCS \begin{definition}
\begin{defn}
The \emph{transformation $\mu_M$ of computable measure $\mu$ by monotone machine $M$} is defined by 
$$\mu_M(\sigma):= \mu(\llbracket \{ \rho: \exists \sigma' \succcurlyeq \sigma ((\rho,\sigma') \in M ) \} \rrbracket ).$$   
%TOCS \end{definition}
\end{defn}

\subsubsection*{Characterizations of $\mathcal{M}$}
For every given semicomputable semimeasure $\nu$, one can obtain a machine $M$ that transforms the uniform measure $\lambda$ to $\nu$. Together with the straightforward converse that every function $\lambda_M$ defines a semicomputable semimeasure, this gives a characterization of the class $\mathcal{M}$ of semicomputable semimeasures as

\begin{equation}\label{eq:charuni}
\mathcal{M}=\{ \lambda_M \}_M,
\end{equation}
where $\{ \lambda_M \}_M$ is the class of functions $\lambda_M$ for all monotone machines $M$. 

A proof of this fact by a construction of an $M$ that transforms $\lambda$ to given $\nu$ was first outlined by Levin in \cite[Theorem 3.2]{ZvoLev70rms}. (Also see \cite[Theorem 4.5.2]{LiVit08}.) Moreover, it can be deduced from \cite[Theorem 3.1(b), 3.2]{ZvoLev70rms} that $\mathcal{M}$ can be characterized as the class of transformations of computable measures other than $\lambda$.  Namely, we have that $ \mathcal{M}$ coincides with $\{ \mu_M \}_M$ for any computable $\mu$ that is continuous.

A detailed construction to prove the characterization \eqref{eq:charuni} was published by Day \cite[Theorem 4(ii)]{Day11tams}. (Also see \cite[Theorem 3.16.2(ii)]{DowHir10}.) The following proof of the case for any continuous computable measure is an adaptation of this construction.

%TOCS \begin{theorem}[Levin]\label{thrm:gen}
\begin{thrm}[Levin]\label{thrm:gen}
For every continuous computable measure $\mu$, there is for every semicomputable semimeasure $\nu$ a monotone machine $M$ such that $\nu=\mu_M$.
%TOCS \end{theorem}
\end{thrm}
%TOCS \begin{proof}
\begin{proof}\renewcommand{\qedsymbol}{}
Let $\nu$ be any semicomputable semimeasure, with uniformly computable approximation functions $f_t$. We construct in stages $s= \langle \sigma, t \rangle$ a monotone machine $M$ that transforms $\mu$ into $\nu$. Let $D_s(\sigma):=\{ \rho \in \mathbb{B}^* : ( \rho,\sigma ) \in M_s \}$.
\end{proof}
%TOCS \begin{constr}
\begin{proof}[Construction]\renewcommand{\qedsymbol}{}
Let $M_0 := \emptyset$.

At stage $s= \langle \sigma, t \rangle$, if $\mu(\llbracket D_{s-1}(\sigma)\rrbracket)=f_t(\sigma)$ then let $M_s := M_{s-1}$.

Otherwise, first consider the case $\sigma \neq \epsilon$. By Lemma 1 in \cite{Day11tams} there is a set $R \subseteq \mathbb{B}^s$ of \emph{available} strings of length $s$ such that $\llbracket  R\rrbracket=\llbracket D_{s-1}(\sigma^-)\rrbracket \setminus (\llbracket D_{s-1}(\sigma^- 0)\rrbracket \cup \llbracket D_{s-1}( \sigma^- 1)\rrbracket)$. Denote $x:=\mu(\llbracket R \rrbracket)$, the amount of measure available for descriptions for $\sigma$, which equals $\mu(\llbracket D_{s-1}(\sigma^-)\rrbracket) - \mu(\llbracket D_{s-1}(\sigma^- 0)\rrbracket)- \mu(\llbracket D_{s-1}(\sigma^- 1)\rrbracket)$ because we ensure by construction that $\llbracket D_{s-1}(\sigma^-)\rrbracket \supseteq\llbracket D_{s-1}(\sigma^- 0)\rrbracket \cup \llbracket D_{s-1}(\sigma^- 1)\rrbracket$ and $\llbracket D_{s-1}(\sigma^- 0)\rrbracket \cap \llbracket D_{s-1}(\sigma^- 1)\rrbracket = \emptyset$. Denote $y:= f_t(\sigma)-\mu(\llbracket D_{s-1}(\sigma)\rrbracket)$, the amount of measure the current descriptions fall short of the latest approximation of $\nu( \sigma)$. We collect in the auxiliary set $A_s$ a number of available strings from $R$ such that $\mu(\llbracket A_s\rrbracket)$ is maximal while still bounded by $\min\{x,y\}$. 

If $\sigma = \epsilon$, then denote $y:= f_t(\epsilon)-\mu(\llbracket D_{s-1}(\epsilon)\rrbracket)$. Collect in $A_s$ a number of available strings from $R \subseteq \mathbb{B}^s$ with $\llbracket R\rrbracket=\mathbb{B}^\omega \setminus \llbracket D_{s-1}(\epsilon)\rrbracket$ such that $\mu(\llbracket A_s\rrbracket)$ is maximal but bounded by $y$. 

Put $M_s:= M_{s-1} \cup \{ ( \rho, \sigma ) : \rho \in A_s \}$.
%TOCS\end{constr}
\end{proof}

%TOCS \begin{veri}
\begin{proof}[Verification]
The verification of the fact that $M$ is a monotone machine is identical to that in \cite{Day11tams}. 

It remains to prove that $\mu_M(\sigma) = \nu(\sigma )$ for all $\sigma \in \mathbb{B}^*$. Since by construction $\llbracket D_s(\sigma') \rrbracket \subseteq \llbracket D_s(\sigma) \rrbracket$ for any $\sigma' \succcurlyeq \sigma$, we have that $\mu_{M_s}(\sigma) =  \mu(\cup_{\sigma' \succcurlyeq \sigma} \llbracket D_s(\sigma') \rrbracket) =  \mu(\llbracket D_s(\sigma) \rrbracket)$. Hence $\mu_M(\sigma) = \lim_{s \rightarrow \infty} \mu(\llbracket D_s(\sigma) \rrbracket)$, and our objective is to show that $\lim_{s \rightarrow \infty} \mu(\llbracket D_s(\sigma)\rrbracket)=\nu( \sigma )$. To that end it suffices to demonstrate that for every $\delta>0$ there is some stage $s_0$ where $\mu(\llbracket D_{s_0}(\sigma)\rrbracket)>\nu( \sigma ) - \delta$. We prove this by induction.

For the base step, let $\sigma=\epsilon$. Choose positive $\delta' < \delta$. There will be a stage $s_0=\langle \epsilon,t_0 \rangle$ where $f_{t_0}(\epsilon)>\nu(\epsilon)-\delta'$, and (since $\mu$ is continuous) $\mu(\llbracket \rho\rrbracket)\leq \delta - \delta'$ for all $\rho \in \mathbb{B}^{s_0}$. Then, if not already $\mu(\llbracket D_{s_0-1}(\epsilon)\rrbracket)> \nu(\epsilon)-\delta$, the latter guarantees that the construction will select a number of available strings in $A_{s_0}$ such that $\nu(\epsilon)-\delta < \mu(\llbracket D_{s_0-1}(\epsilon)\rrbracket)+ \mu(\llbracket A_s \rrbracket) \leq f_{t_0}(\epsilon)$. It follows that  $\mu(\llbracket D_{s_0}(\epsilon)\rrbracket)= \mu(\llbracket D_{s_0-1}(\epsilon)\rrbracket)+\mu(\llbracket A_s \rrbracket)> \nu(\epsilon)-\delta$ as required.
 
For the inductive step, let $\sigma \neq \epsilon$, and denote by $\sigma'$ the one-bit extension of $\sigma^-$ with $\sigma' \mid \sigma$. Choose positive $\delta' < \delta$. By induction hypothesis, there exists a stage $s_0'$ such that $\mu(\llbracket D_{s_0'}(\sigma^-)\rrbracket)> \nu( \sigma^-)-\delta'$. At this stage $s_0'$, we have 

\begin{align*}
\mu(\llbracket D_{s_0'}(\sigma^-)\rrbracket)-\mu(\llbracket D_{s_0'}(\sigma')\rrbracket) &\geq \mu(\llbracket D_{s_0'}(\sigma^-)\rrbracket-\nu(\sigma')
\\ &> \nu(\sigma^-)-\delta'-\nu(\sigma') 
\\ &\geq \nu(\sigma)-\delta',
\end{align*}
where the last inequality follows from the semimeasure property $\nu(\sigma^-)\geq \nu(\sigma)+\nu(\sigma')$. There will be a stage $s_0=\langle \sigma,t_0 \rangle \geq s_0'$ with $f_{t_0}(\sigma)>\nu(\sigma)-\delta'$ and $\mu(\llbracket \rho\rrbracket)\leq \delta - \delta'$ for all $\rho \in \mathbb{B}^{s_0}$. Clearly, $\min \{\mu(\llbracket D_{s_0}(\sigma^-)\rrbracket)-\mu(\llbracket D_{s_0}(\sigma')\rrbracket),f_{t_0}(\sigma) \} > \nu(\sigma)-\delta'$. Then, as in the base case, if not already $\mu(\llbracket D_{s_0-1}(\sigma)\rrbracket)> \nu(\sigma)-\delta$, the construction selects a number of available descriptions such that $\mu(\llbracket D_{s_0}(\sigma)\rrbracket)> \nu(\sigma)-\delta$ as required. %TOCS \qed
%TOCS \end{veri}
\end{proof}

%TOCS \begin{corollary}\label{cor:charcont}
\begin{cor}\label{cor:charcont}
For every continuous computable measure $\mu$,
$$\{ \mu_M \}_M = \mathcal{M} .$$
%TOCS \end{corollary}
\end{cor}

\subsection{Prefix-free machines and discrete semimeasures}\label{subsec:discrsemi}
The notions of a semicomputable \emph{discrete} semimeasure on the finite strings and a \emph{prefix-free} machine can be traced back to Levin \cite{Lev74pit} and G\'acs \cite{Gac74smd}, and independently Chaitin \cite{Cha75jacm}.

%TOCS\begin{definition}
\begin{defn}
A \emph{semicomputable discrete semimeasure} is a semicomputable function $P: \mathbb{B}^* \rightarrow \mathbb{R}^{\geq 0}$ such that $\sum_{\sigma \in \mathbb{B}^*} P(\sigma) \leq 1$.
\end{defn}
%TOCS\end{definition}

%TOCS\begin{definition}
\begin{defn}
A \emph{prefix-free machine} is a partial computable function $T:\mathbb{B}^* \rightarrow \mathbb{B}^*$ with prefix-free domain.
\end{defn}
%TOCS\end{definition}

%TOCS\begin{definition}
\begin{defn}
The \emph{transformation of computable measure $\mu$ by prefix-free machine $T$} is the semicomputable discrete semimeasure $Q^\mu_T: \mathbb{B}^* \rightarrow [0,1]$ defined by 
$$Q^\mu_T(\sigma):= \mu(\llbracket \{ \rho: (\rho,\sigma) \in T \} \rrbracket).$$ 
\end{defn}
%TOCS\end{definition}

Let $\mathcal{P}$ denote the class of all semicomputable discrete semimeasures. Analogous to class $\mathcal{M}$ and the monotone machines, class $\mathcal{P}$ is characterized as all prefix-free machine transformations of $\mu$, for any continuous computable $\mu$. The fact that every $P$ can be obtained as a transformation of $\lambda$ is usually inferred from the effective version of Kraft's inequality (e.g., \cite[p.\ 130]{DowHir10},  \cite[Exercise 2.2.23]{Nie09}). However, we can easily prove the general case in a direct manner by a much simplified version of the construction for Theorem \ref{thrm:gen}.

%\begin{proposition}\label{propo:gendiscr}
\begin{propo}\label{propo:gendiscr}
For every continuous computable measure $\mu$, there is for every semicomputable discrete semimeasure $P$ a prefix-free machine $T$ such that $P=Q^\mu_T$.
%\end{proposition}
\end{propo}
%TOCS \begin{proof}%\renewcommand{\qedsymbol}{}
\begin{proof}\renewcommand{\qedsymbol}{}
Let $P$ be any semicomputable discrete semimeasure, with uniformly computable approximation functions $f_t$. We construct a prefix-free machine $T$ in stages $s= \langle \sigma, t \rangle$. Let $D_s(\sigma)=\{ \rho \in \mathbb{B}^* : ( \rho,\sigma ) \in T_s \}$. 
\end{proof}
%TOCS \begin{constr}
\begin{proof}[Construction]\renewcommand{\qedsymbol}{}
Let $T_0 = \emptyset$.

At stage $s= \langle \sigma, t \rangle$, if $\mu(\llbracket D_{s-1}(\sigma)\rrbracket)=f_t(\sigma)$ then let $T_s := T_{s-1}$.

Otherwise, let the set $R \subseteq \mathbb{B}^s$ of \emph{available} strings be such that $\llbracket  R\rrbracket= \mathbb{B}^\omega \setminus \llbracket \cup_{\tau \in \mathbb{B}^*} D_{s-1}(\tau) \rrbracket$. Collect in the auxiliary set $A_s$ a number of available strings $\rho$ from $R$ with $\sum_{\rho \in A_s} \mu(\llbracket\rho\rrbracket)$ maximal but bounded by $f_t(\sigma)-\mu(\llbracket D_{s-1}(\sigma)\rrbracket)$, the amount of measure the current descriptions fall short of the latest approximation of $P(\sigma)$. Put $T_s:= T_{s-1} \cup \{ ( \rho, \sigma ) : \rho \in A_s \}$.
%TOCS \end{constr}
\end{proof}

\begin{proof}[Verification]
%TOCS \begin{veri}
It is immediate from the construction that $\cup_{\sigma \in \mathbb{B}^*} D_s(\sigma)$ is prefix-free at all stages $s$, so $T = \lim_{s \rightarrow \infty} T_s$ is a prefix-free machine. To show that $Q^\mu_T(\sigma)=\lim_{s \rightarrow \infty} \mu(\llbracket D_s(\sigma)\rrbracket)$ equals $P(\sigma)$ for all $\sigma \in \mathbb{B}^*$, it suffices to demonstrate that for every $\delta>0$ there is some stage $s_0$ where $\mu(\llbracket D_{s_0}(\sigma)\rrbracket)>P(\sigma) - \delta$. 

Choose positive $\delta' < \delta$. Wait for a stage $s_0=\langle \sigma,t_0 \rangle$ with $\mu(\llbracket \rho \rrbracket)\leq \delta - \delta'$ for all $\rho \in \mathbb{B}^{s_0}$ and $f_{t_0}(\sigma)>P(\sigma)-\delta'$. Clearly, the available $\mu$-measure

\begin{align*}
\mu(\llbracket R \rrbracket) &= 1- \sum_{\tau \in \mathbb{B}^*} \mu(\llbracket  D_{s_0-1}(\tau) \rrbracket)
\\ &\geq 1- \mu(\llbracket  D_{s_0-1}(\sigma) \rrbracket) - \sum_{\tau \in \mathbb{B}^*\setminus \{\sigma\}} P(\tau)  
\\ &\geq P(\sigma) -\mu(\llbracket  D_{s_0-1}(\sigma) \rrbracket) 
\\ &\geq f_{t_0}(\sigma) -\mu(\llbracket  D_{s_0-1}(\sigma) \rrbracket).
\end{align*}

Consequently, if not already $\mu(\llbracket D_{s_0-1}(\sigma)\rrbracket)> P(\sigma)-\delta$, then the construction collects in $A_{s_0}$ a number of descriptions of length $s_0$ from $R$ such that $\mu(\llbracket D_{s_0}(\sigma)\rrbracket)= \mu(\llbracket D_{s_0-1}(\sigma)\rrbracket)+\sum_{\rho \in A_{s_0}}\mu(\llbracket \rho \rrbracket) > P(\sigma)-\delta$ as required. %TOCS \qed
%TOCS \end{veri}
\end{proof}

%TOCS \begin{corollary}\label{cor:chardiscr}
\begin{cor}\label{cor:chardiscr}
For every continuous computable measure $\mu$,
$$\{ Q_T^\mu \}_T = \mathcal{P}. $$
%TOCS \end{corollary}
\end{cor}

\section{The a priori semimeasures}\label{sec:contap}
In this section we show that the class of a priori semimeasures can be characterized as the class of universal transformations of any continuous computable measure. Subsection \ref{subsec:contap} introduces the class of a priori semimeasures. Subsection \ref{subsec:mixt} is an interlude devoted to the representation of the a priori semimeasures as \emph{universal mixtures}. Subsection \ref{subsec:apgen} presents the generalized characterization, and concludes with a brief discussion of how this reflects on the association with foundational principles.

\subsection{A priori semimeasures}\label{subsec:contap}

\subsubsection*{Universal machines}
Let $\{ \rho_e \}_{e \in \mathbb{N}} \subseteq \mathbb{B}^*$ be any computable prefix-free and non-repeating enumeration of finite strings, that will serve as an encoding of some computable enumeration $\{ M_e \}_{e \in \mathbb{N}}$ of all monotone machines. We say that a monotone machine $U$ is \emph{universal (by adjunction)} if for some such
encoding $\{ \rho_e \}_{e \in \mathbb{N}}$, we have for all $\rho,\sigma \in \mathbb{B}^*$ that

\begin{equation*}
 (\rho_e \rho,\sigma) \in U  \Leftrightarrow (\rho,\sigma) \in M_e.
\end{equation*}

By a universal machine we will mean a machine that is universal by adjunction. Contrast this to \emph{weak universality}%(cf.\ \cite[p.\ 3492-3]{BarDow12ptrs})
, which is the more general property that for all $M$ there is a $c_M \in \mathbb{N}$ such that 
\begin{equation*}
 (\rho,\sigma) \in M  \Rightarrow \exists \rho'\big(|\rho'|<|\rho|+c_M \ \& \ (\rho,\sigma) \in U\big).
\end{equation*}

\subsubsection*{A priori semimeasures}
We call a transformation by a universal machine a \emph{universal transformation}. The a priori semimeasures are the universal transformations of the uniform measure.

%TOCS\begin{definition}\label{def:cap}
\begin{defn}\label{def:cap}
An \emph{a priori semimeasure} is defined by 
$$\lambda_U(\sigma) := \lambda(\llbracket \{ \rho: \exists \sigma' \preccurlyeq \sigma ((\rho,\sigma') \in U ) \} \rrbracket ) $$
for universal monotone machine $U$.
%TOCS\end{definition}
\end{defn}

Let $\mathcal{A}$ denote the class $\{ \lambda_U \}_U$ of a priori semimeasures. The next result implies that every element of $\mathcal{A}$ can also be obtained as the transformation of $\lambda$ by a machine that is \emph{not} universal.

%\begin{proposition}\label{propo:univnuniv}
\begin{propo}\label{propo:univnuniv}
For every continuous computable measure $\mu$, there is for every semicomputable semimeasure $\nu$ a \emph{non-universal} monotone machine $M$ such that $\nu=\mu_M$.
%TOCS \end{proposition}
\end{propo}
\begin{proof}
Let $U$ be an arbitrary universal machine. We will adapt the construction of Theorem \ref{thrm:genap} of a machine $M$ with $\mu_M = \nu$ in such a way that for every constant $c\in \mathbb{N}$ there is a $\sigma$ such that for some $\rho'$ with $(\rho',\sigma) \in U$, we have that $| \rho| > |\rho'| + c$ for all $\rho$ with $(\rho,\sigma) \in M$. This ensures that $M$ is not even weakly universal.
\end{proof} 

%TOCS \begin{constr}
\begin{proof}[Construction]\renewcommand{\qedsymbol}{}
The only change to the earlier construction is that at stage $s$ we try to collect available strings of length $l_s$, where $l_s$ is defined as follows. Let $l_0 = 0$. For $s=\langle \sigma, t\rangle$ with $t>0$, let $l_s = l_{s-1}+1$. In case $s=\langle \sigma, 0\rangle$, enumerate pairs in $U$ until a pair $(\rho', \sigma)$ for some $\rho'$ is found. Let $l_s := \max\{l_{s-1}+1,|\rho'|+s \}$.
%TOCS\end{constr}
\end{proof}

%TOCS \begin{veri}
\begin{proof}[Verification]
The verification that $\mu_M = \nu$ proceeds as before. In addition, the construction guarantees that for every $c \in \mathbb{N}$, we have for $\sigma$ with $c=\langle \sigma, 0 \rangle$ that $| \rho| > |\rho'| + c$ for the first enumerated $\rho'$ with $(\rho',\sigma) \in U$ and all $\rho$ with $(\rho,\sigma) \in M$. %TOCS \qed
%TOCS \end{veri}
\end{proof}

\subsection{Universal mixtures}\label{subsec:mixt}
Every element of $\mathcal{A}$ is equal to a \emph{universal mixture}

\begin{equation}\label{eq:mix}
\xi_W(\cdot) := \sum_{i \in \mathbb{N}} W(i) \nu_i(\cdot)
\end{equation}
for some effective enumeration $\{\nu_i\}_{i \in \mathbb{N}}=\mathcal{M}$ of all semicomputable semimeasures, and some semicomputable \emph{weight function} $W: \mathbb{N} \rightarrow [0,1]$ that satisfies $\sum_{i \in \mathbb{N}} W(i) \leq 1$ and $W(i)>0$ for all $i$. Conversely, one can show that every universal mixture equals $\lambda_U$ for some universal machine $U$ \cite{WooSunHut13inproc}.

Let $\mathcal{U}$ denote the elements $\kappa$ of $\mathcal{M}$ that are \emph{universal} in the sense that they  \emph{dominate} every other semicomputable semimeasure. That is, for such $\kappa \in \mathcal{U}$ there is for every $\nu \in \mathcal{M}$ a constant $c_\nu \in \mathbb{N}$, depending only on $\kappa$ and $\nu$, such that $ \kappa(\sigma) \geq c^{-1}_\nu \nu(\sigma)$ for all $\sigma \in \mathbb{B}^*$. It is clear from the mixture form of the a priori semimeasures that $\mathcal{A} \subseteq \mathcal{U}$. This inclusion is strict: not all universal elements are of the form $\lambda_U$ (equivalently, mixtures). For instance, $\xi_W(\epsilon)<1$ for all $W$ because $\nu(\epsilon)<1$ for some $\nu \in \mathcal{M}$, but we can obviously define a universal $\kappa \in \mathcal{M}$ with $\kappa(\epsilon)=1$.

We can strengthen the above statement of the equivalence of the a priori semimeasures and the universal mixtures by requiring a computable weight function $W$ over a fixed enumeration $\{\nu_i\}_{i \in \mathbb{N}}$, as follows. 

First, let us call an enumeration $\{ \nu_i \}_{i \in \mathbb{N}}$ of all semicomputable semimeasures \emph{acceptable} if it is generated from an enumeration $\{M_i\}_i$ of all monotone Turing machines by the procedure of Theorem \ref{thrm:gen}, i.e., $\nu_i = \lambda_{M_i}$. This terminology matches that of the definition of \emph{acceptable numberings} of the partial computable functions \cite[p.\ 41]{Rog67}. Every effective listing of all Turing machines yields an acceptable numbering. Importantly, any two acceptable numberings differ only by a computable permutation \cite{Rog58jsl}; in our case, for any two acceptable enumerations $\{ \nu_i \}_i$ and $\{ \bar{\nu}_i \}_i$ there is a computable permutation $f: \mathbb{N}\rightarrow \mathbb{N}$ of indices such that $\bar{\nu}_i = \nu_{f(i)}$. 

Furthermore, let us call a semicomputable weight function $W$ \emph{proper} if $\sum_i W(i) = 1$; this implies that $W$ is computable. 

Then we can show that for any acceptable enumeration of all semicomputable semimeasures, all elements in $\mathcal{A}$ are expressible as some mixture with a \emph{proper} weight function over \emph{this} enumeration. 

\begin{propo}
For every acceptable enumeration $\{ \nu_i \}_i$ of $\mathcal{M}$, every element in $\mathcal{A}$ is equal to $\xi_W(\cdot) = \sum_i W(i) \nu_i(\cdot)$ for some proper $W$.
\end{propo}
\begin{proof}
Given $\lambda_U \in \mathcal{A}$, with enumeration $\{M_i\}_i$ of all monotone machines corresponding to $U$. We know that $\lambda_U$ is equal to $\bar{\xi}_W(\cdot) = \sum_i W(i) \bar{\nu}_i(\cdot)$ for acceptable enumeration $\{ \bar{\nu}_i \}_i=\{ \lambda_{M_i}\}_i$ of $\mathcal{M}$ and semicomputable weight function $W$. First we show that $\bar{\xi}_W$ is equal to $\xi_{W'}(\cdot) = \sum_i W'(i) \nu_i(\cdot)$ for given acceptable enumeration $\{ \nu_i \}_i$ and semicomputable $W'$; then we show that it is also equal to $\xi_{W''}(\cdot) = \sum_i W''(i) \nu_i(\cdot)$ for proper $W''$. 

Since enumerations $\{ \nu_i \}_i$ and $\{ \bar{\nu}_e \}_e$ are both acceptable, there is a 1-1 computable $f$ such that $\bar{\nu}_i = \nu_{f(i)}$. Then

\begin{align*}
\sum_i W(i) \bar{\nu}_i(\cdot) &= \sum_i W(i) \nu_{f(i)}(\cdot)
\\ &= \sum_i W(f^{-1}(i))\nu_i(\cdot)
\\ &= \sum_i W'(i)\nu_i(\cdot),
\end{align*}
with $W': i \mapsto W(f^{-1}(i))$.

We proceed with the description of a proper $W''$. The idea is to have $W''$ assign to each $i$ a positive computable weight that does not exceed $W'(i)$, additional computable weight to the index of a single suitably defined semimeasure in order to regain the original mixture, and all of the remaining weight to an ``empty'' semimeasure. 

Let $q \in \mathbb{Q}$ be such that $\xi_{W'}(\epsilon) < q < 1$, and let $c$ be such that $\sum_i 2^{-i-c} < 1-q$. Let $W'_{0}(i)$ denote the first approximation of semicomputable $W'(i)$ that is positive. We now define computable $g: \mathbb{N}\rightarrow\mathbb{Q}$ by 
$$ g(i)= \min\{ 2^{-i-c}, W'_{0}(i)\}.$$
Clearly, $\sum_i g(i) < 1-q$. Moreover, $\sum_i g(i)$ is computable because for any $\delta > 0$ we have a $j \in \mathbb{N}$ with $\sum_{i>j}2^{-i-c} < \delta$, hence $\sum_{i \leq j} g(i) < \sum_i g(i) < \sum_{i \leq j} g(i) + \delta$.

Next, define $\pi(\cdot) = q^{-1} \sum_i \left( W'(i)-g(i) \right) \nu_i(\cdot)$. This is a semimeasure because $\pi(\epsilon) \leq q^{-1} \xi_{W'}(\epsilon) < q^{-1}q = 1$. Let $k$ be such that $\nu_k = \pi$, and let $l$ be such that $\nu_l$ is the ``empty'' semimeasure with $\nu(\sigma)=0$ for all $\sigma \in \mathbb{B}^*$ (both indices exist even if we cannot effectively find them).

Finally, we define $W''$ by

$$W''(i) = \begin{cases}
   g(i) & \text{if } i \neq k,l \\
   g(i) + q       & \text{if }i=k \\
   1-q-\sum_{j \neq l} g(j) & \text{if }i=l  
  \end{cases}.$$
  
weight function $W''$ is computable and indeed proper, and  
\begin{align*}
\sum_i W''(i) \nu_i(\cdot) &= \sum_i g(i)\nu_i(\cdot) + q  \nu_k(\cdot) + 0
\\ &= \sum_i g(i)\nu_i(\cdot) + \sum_i \left( W'(i)-g(i) \right) \nu_i(\cdot)
\\ &= \sum_i W'(i)\nu_i(\cdot).
\end{align*}
\end{proof}

%\texttt{In termen van machines?}
As a kind of converse, we can derive that any universal mixture is also equal to a universal mixture with a \emph{universal} weight function, i.e., a weight function $\bar{W}$ such that for all other $W$ there is a $c_W$ with $\bar{W}(i) \geq c_W^{-1} W(i)$ for all $i$. 

\begin{propo}
For every acceptable enumeration $\{ \nu_i \}_i$ of $\mathcal{M}$, every element in $\mathcal{A}$ is equal to $\xi_{\bar{W}}(\cdot) = \sum_i \bar{W}(i) \nu_i(\cdot)$ for some universal $\bar{W}$.
\end{propo}
\begin{proof}
By the above proposition we know that any given element in $\mathcal{A}$ equals $\xi_W = \sum_i W(i)\nu_i$ for some (computable) $W$ over given $\{ \nu_i \}_i$. Let $k$ be such that $\nu_k = \sum_i 2^{-K(i)} \nu_i$, with $K(i)$ the prefix-free Kolmogorov complexity (via some universal prefix-free machine $U$) of the $i$-th lexicographically ordered string; $2^{-K(\cdot)}$ is a universal weight function. Define 
$$\bar{W}(i) = \begin{cases}
   W(i)+W(k)\cdot 2^{-K(i)} & \text{if } i \neq k \\
   W(k)\cdot2^{-K(i)}      & \text{if }i=k  
  \end{cases},$$
which is a weight function because $\sum_i \bar{W}(i) < \sum_{i\neq k} W(i) + W(k)=\sum_i W(i)$. Moreover, $\bar{W}$ is universal because $2^{-K(\cdot)}$ is, and
\begin{align*}
\sum_i \bar{W}(i) \nu_i(\cdot) &=  \sum_{i \neq k} W(i)\nu_i(\cdot) + W(k) \sum_i 2^{-K(i)}\nu_i(\cdot) 
\\ &= \sum_i W(i)\nu_i(\cdot).
\end{align*}\end{proof} 

Hutter \cite[p.\ 102-03]{Hut05} argues that a universal mixture with weight function $2^{-K(i)}$ is \emph{optimal} among all universal mixtures, essentially because this weight function is universal. The above result shows that this optimality is meaningless: \emph{every} universal mixture can be represented so as to have a universal weight function.

\subsection{The generalized characterization}\label{subsec:apgen}
We are now ready to show that the universal transformations of any continuous computable measure $\mu$ yield the same class $\mathcal{A}$ of a priori semimeasures. A minor caveat is that we will need to restrict the universal machines $U$ to those machines with associated encodings $\{ \rho_e \}_e$ that do not receive measure 0 from $\mu$: so $\mu(\llbracket \rho_e \rrbracket)>0$ for all $e \in \mathbb{N}$. Call (the associated encodings of) those machines \emph{compatible} with measure $\mu$. This is clearly no restriction for measures that give positive probability to every finite string (such as the uniform measure): all machines are compatible to such measures.

We will prove:

%TOCS \begin{theorem}\label{thrm:genap}
\begin{thrm}\label{thrm:genap}
Let $\mu, \tilde{\mu}$ be continuous computable measures. For universal machine $U$ that is compatible with $\mu$, there is universal machine $\widetilde{U}$ such that $\mu_U = \tilde{\mu}_{\widetilde{U}}$.
%TOCS \end{theorem}
\end{thrm}

It follows that $\{\mu_U \}_U = \{ \tilde{\mu}_{\widetilde{U}} \}_{\widetilde{U}}$ for any two continuous computable $\mu$ and $ \tilde{\mu}$,
with $U$ ranging over those universal machines compatible with $\mu$ and $\widetilde{U}$ over those universal machines compatible with $\tilde{\mu}$. In particular, since $\lambda$ is itself a continuous computable measure, we have that $\{\mu_U \}_U = \mathcal{A}$.

Our proof strategy is to expand the approach taken in \cite{WooSunHut13inproc} to show the coincidence of the a priori semimeasures and the universal mixtures. Let us first derive the fact that a universal transformation of $\mu$ is an a priori semimeasure.

%\begin{proposition}\label{propo:genuniv}
\begin{propo}\label{propo:genuniv}
Let $\mu$ be a continuous computable measure and universal machine $U$ compatible with $\mu$. Then $\mu_U \in \mathcal{A}$.
%TOCS \end{proposition} 
\end{propo}

The proof rests on a fixed-point lemma that is a refined version of Corollary \ref{cor:charcont}. For given encoding $\{ \rho_e \}_e$, define $\mu^e(\cdot):= \mu(\cdot \mid \llbracket \rho_e \rrbracket)$ for any $e \in \mathbb{N}$. Here the conditional measure $\mu( \llbracket \tau \rrbracket \mid \llbracket \sigma \rrbracket ) := \frac{\mu(\llbracket \sigma \tau \rrbracket)}{\mu(\llbracket \sigma \rrbracket)}$ for any $\sigma, \tau \in \mathbb{B}^*$. 

%TOCS \begin{lemma}\label{lmm:condequiv}
\begin{lmm}\label{lmm:condequiv}
Given encoding $\{ \rho_e \}_{e \in \mathbb{N}}$ of the monotone machines as above. For every continuous computable measure $\mu$, 
$$\{ \mu^e_{M_e} \}_e = \mathcal{M}.$$
%\end{lemma}
\end{lmm}

\begin{proof}
Let $\nu$ be any semicomputable semimeasure. Since $\mu^e$ is obviously a computable measure for every $e \in \mathbb{N}$, by the construction of Theorem \ref{thrm:gen} we obtain for every $e$ a monotone machine $M$ with $\nu=\mu^e_M$. Indeed, there is a total computable function $g: \mathbb{N} \rightarrow \mathbb{N}$ that for given $e$ retrieves an index $g(e)$ in the given enumeration $\{ M_e \}_{e \in \mathbb{N}}$ such that $\nu=\mu^e_{M_{g(e)}}$. But by the Recursion Theorem, there must be a fixed point $\hat{e}$ such that $M_{g(\hat{e})}=M_{\hat{e}}$, hence $\mu^{\hat{e}}_{M_{\hat{e}}}=\mu^{\hat{e}}_{M_{g(\hat{e})}}$. 

This shows that for every $\nu$ there is an index $e$ such that $\nu=\mu^e_{M_e}$. Conversely, the function $\mu^e_{M_e}$ is a semicomputable semimeasure for every $e$. %TOCS \qed
\end{proof}

%TOCS\begin{proof}[of Proposition \ref{propo:genuniv}]
\begin{proof}[Proof of Proposition \ref{propo:genuniv}]
Given continuous computable $\mu$ and universal $U$ compatible with $\mu$. We write out

\begin{align*}
\mu_U (\sigma) &= \mu(\llbracket \{ \rho : \exists \sigma' \succcurlyeq \sigma (( \rho, \sigma') \in U) \}\rrbracket)
\\ &= \sum_e \mu(\llbracket \{ \rho_e \rho : \exists \sigma' \succcurlyeq \sigma (( \rho, \sigma') \in M_e) \}\rrbracket)
\\ &=  \sum_e \mu(\llbracket \rho_e \rrbracket) \mu(\llbracket \{ \rho : \exists \sigma' \succcurlyeq \sigma (( \rho, \sigma') \in M_e) \}\rrbracket \mid \llbracket   \rho_e   \rrbracket )
\\ &= \sum_e \mu(\llbracket \rho_e \rrbracket) \mu^e_{M_e}(\sigma ).
\end{align*}

Lemma \ref{lmm:condequiv} tells us that the $\mu^e_{M_e}$ range over all elements in $\mathcal{M}$. Moreover, $W(e):= \mu(\llbracket \rho_e \rrbracket)$ is a weight function because $\{\rho_e\}_e$ is prefix-free and $U$ is compatible with $\mu$, so $\mu_U$ is a universal mixture. %TOCS\qed
\end{proof}

We now proceed to prove that every universal transformation of $\mu$ indeed equals some universal transformation of $\tilde{\mu}$.

%TOCS \begin{proof}[of Theorem \ref{thrm:genap}]
\begin{proof}[Proof of Theorem \ref{thrm:genap}]
Given continuous computable $\mu$ and $\tilde{\mu}$, and universal $U$ compatible with $\mu$. Write out as before
 $$\mu(\sigma) =\sum_e \mu(\llbracket \rho_e \rrbracket) \mu^e_{M_e}(\sigma ).$$ 
Note that the function 

$$P(\sigma) = \begin{cases}
   \mu(\llbracket \sigma \rrbracket) & \text{if } \sigma = \rho_e \text{ for some }e\in \mathbb{N} \\
   0       & \text{otherwise}
  \end{cases}$$
is a semicomputable discrete semimeasure. Hence by Proposition \ref{propo:gendiscr} we can construct a prefix-free machine $T$ that transforms $\tilde{\mu}$ into $P$: so $Q^{\tilde{\mu}}_T = P$. Denote $n_e := \#\{\tau: (\tau,\rho_e)\in T \}$ the number of $T$-descriptions of $\rho_e$, and let $\langle \cdot,\cdot \rangle: \mathbb{N} \times \mathbb{N} \rightarrow \mathbb{N}$ be a partial computable pairing function that maps the pairs $(e,i)$ with $i < n_e$ onto $\mathbb{N}$. Let $\tilde{\rho}_{\langle e,i \rangle}$ be the $i$-th enumerated $T$-description of $\rho_e$. We then have

\begin{align*}
\sum_e \mu(\llbracket \rho_e \rrbracket) \mu^e_{M_e}(\sigma )
&= \sum_e Q^{\tilde{\mu}}_T(\rho_e) \mu^e_{M_e}(\sigma )
\\  &= \sum_{e}\sum_{i < n_e} \tilde{\mu}(\llbracket \tilde{\rho}_{\langle e,i \rangle} \rrbracket) \mu^e_{M_e}(\sigma ).
\end{align*}

Write $\mu^{\tilde{d}}$ for $\mu(\cdot \mid \llbracket \tilde{\rho}_d \rrbracket)$. Now for every $\langle e,i \rangle$ for which $\tilde{\rho}_{\langle e,i \rangle}$ becomes defined we can run the construction of Theorem \ref{thrm:gen} on $\tilde{\mu}^{\langle \widetilde{e,i} \rangle}$ and $\mu^e_{M_e}$. In this way we obtain an enumeration of machines $\{ \widetilde{M}_d \}_d$ such that $\tilde{\mu}^{\langle \widetilde{e,i} \rangle}_{\widetilde{M}_{\langle e,i \rangle}}=\mu^e_{M_e}$ (with $i < n_e$) for all $e$. Then 

$$ \sum_{e}\sum_{i < n_e} \tilde{\mu}(\llbracket \tilde{\rho}_{\langle e,i \rangle} \rrbracket) \mu^e_{M_e}(\sigma ) = \sum_d   \tilde{\mu}(\llbracket \tilde{\rho}_d \rrbracket)  \tilde{\mu}^{\tilde{d}}_{\tilde{M}_d}(\sigma),$$
which we can rewrite to $\tilde{\mu}_{\widetilde{U}}(\sigma)$, defining $\widetilde{U}$ by $(\tilde{\rho}_d \rho,\sigma) \in \widetilde{U} :\Leftrightarrow (\rho,\sigma) \in \widetilde{M}_d$.

It remains to verify that $\widetilde{U}$ is in fact universal. Namely, we cannot take for granted that $\{ \widetilde{M}_d \}_{d \in \mathbb{N}}$ is an enumeration of \emph{all} machines, 
%(while the $\mu^e_{M_e}$ give all elements in $\mathcal{M}$, any such semimeasure can be obtained as the $\mu$-transformation of multiple different machines)
hence it is not clear that $\widetilde{U}$ is universal.\footnote{This is also an (overlooked) issue in the original proof in \cite[Lemma 4]{WooSunHut13inproc}. It is easily resolved by the same approach we take below, where it is immediate that for universal $V$ there is $e$ with $\lambda_V = \nu_e$.} Note that it is enough if there were a single universal machine $\widetilde{U}'$ in $\{ \widetilde{M}_d \}_{d \in \mathbb{N}}$, but even that is not obvious (by Proposition \ref{propo:univnuniv} we know that for all continuous $\mu$ there are for any universal $U$ \emph{non}-universal $M$ such that $\mu_M = \mu_U$).

However, there is a simple patch to the enumeration that guarantees this fact. Namely, given an arbitrary universal machine $V$, we may simply put $\widetilde{M}_d := V$ at some $d=\langle e,i \rangle$ where it so happens that $\tilde{\mu}^{{\langle \widetilde{e,i} \rangle}}_V = \mu^{e}_{M_e}$. (We cannot effectively find this $d$, but it is finite information so if this $d$ exists then so does the patched enumeration.) 

Our final objective is then to show that $\tilde{\mu}^{\langle \widetilde{e,i} \rangle}_V = \mu^e_{M_e}$ for some $e,i$. Define computable $g: \mathbb{N} \rightarrow \mathbb{N}$ by $\mu^e_{M_{g(e)}} = \tilde{\mu}^{\langle\widetilde{ \hat{e},0 }\rangle}_V$. Since $Q^{\tilde{\mu}}_T(\rho_e) > 0$ for each $e$, the string $\widetilde{\rho}_{\langle e,0 \rangle}$ is defined for each $e$. Hence $\tilde{\mu}^{\langle\widetilde{ \hat{e},0 }\rangle}_V$ is defined, and function $g$, that retrieves the index $g(e)$ of a machine that transforms $\mu^e$ to this semimeasure, is total. Then by the Recursion Theorem there is index $\hat{e}$ such that $M_{\hat{e}}=M_{g(\hat{e})}$, so $\mu^e_{M_{\hat{e}}} = \mu^e_{M_{g(\hat{e})}}=\tilde{\mu}^{\langle\widetilde{ \hat{e},0 }\rangle}_V$. %TOCS \qed
\end{proof}

%TOCS \begin{corollary}\label{cor:contgenap}
\begin{cor}\label{cor:contgenap}
For continuous computable $\mu$, and $U$ ranging over those universal machines that are compatible with $\mu$,

$$\{\mu_U \}_U = \mathcal{A}.$$
%TOCS \end{corollary}
\end{cor}

\subsubsection*{Discrete a priori semimeasures}
A \emph{universal} prefix-free machine $U$ is defined by
$$(\rho_e\rho,\sigma) \in U \Leftrightarrow (\rho,\sigma) \in T_e $$
for all $\rho, \sigma \in \mathbb{B}^*$ and some computable prefix-free and non-repeating enumeration $\{ \rho_e \}_{e \in \mathbb{N}} \subseteq \mathbb{B}^*$ that serves as an encoding of some computable enumeration $\{ T_e \}_{e \in \mathbb{N}}$ of all prefix-free machines. 

%TOCS \begin{definition}
\begin{defn}
A \emph{discrete a priori semimeasure} is defined by 
$$Q^\lambda_U(\sigma) : = \lambda(\llbracket \{ \rho: (\rho,\sigma) \in U  \} \rrbracket ) $$
for a universal prefix-free machine $U$.
%TOCS \end{definition}
\end{defn}

Let $\mathcal{Q}$ denote the class of all discrete a priori semimeasures. Discrete versions of the above results are derived in an identical manner. Ultimately, we have the following discrete analogue to Corollary \ref{cor:contgenap}.

%\begin{proposition}\label{propo:discrgenap}
\begin{propo}\label{propo:discrgenap}
For continuous computable $\mu$, and $U$ ranging over those prefix-free machines that are compatible with $\mu$,
$$\{Q^\mu_U \}_U= \mathcal{Q}.$$
%TOCS \end{proposition}
\end{propo}

\subsubsection*{Discussion} 
We now return to the association of the function $\lambda_U$ (as well as its discrete counterpart $Q^\lambda_U$) with foundational principles.

First, there is the association with the principle of \emph{insufficient reason} or \emph{indifference}. This is the principle that in the absence of discriminating evidence, probability should be equally distributed over all possibilities. 
Solomonoff writes, 
``If we consider the input sequence to be the `cause' of the observed output sequence, and we consider all input sequences of a given length to be equiprobable (since we have no a priori reason to prefer one rather than the other) then we obtain the present model of induction.'' \cite[p.\ 19]{Sol64ftii}. Also see \cite{LiVit92inproc,RatHut11e}.

Second, there is the association with Occam's razor. Solomonoff writes, ``That [this model] might be valid is suggested by `Occam's razor,' one interpretation of which is that the more `simple' or `economical' of several hypotheses is the more likely \dots\ ---the most `simple' hypothesis being that with the shortest `description.'{''} \cite[p.\ 3]{Sol64ftii}. Also see \cite{Sol97jcss,LiVit08,Hut07tcs,CovTho06,OrtLei11mi}.

Note that so stated, these associations very much rely on the fact that the uniform measure $\lambda$ always assigns larger probability to shorter strings, and equal probability to equal-length strings. This is a unique feature of $\lambda$. The results of this paper, however, imply that the choice of the uniform measure in defining algorithmic probability is only circumstantial: we could pick \emph{any} continuous computable measure, and still obtain, as the universal transformations of \emph{this} measure instead of $\lambda$, the very same class of a priori semimeasures. This suggests that properties derived from the presence of $\lambda$ in the definition are artifacts of a particular choice of characterization rather than an indicative property of algorithmic probability, and hence undermines both associations insofar as they indeed hinge on the uniform measure.

%\bibliography{all}
%\bibliographystyle{abbrv}

\end{document}